\documentclass[11pt]{amsart}
\usepackage{amssymb,amsfonts,amsmath,amsthm}
\usepackage[all,2cell]{xy}
\usepackage{graphicx}
\usepackage{pstricks}
\usepackage[mathscr]{eucal}
\usepackage{amscd}
\usepackage{color}		
\usepackage{epsfig}
\usepackage{enumerate}
\usepackage{indentfirst}
\usepackage{fancyhdr}

\theoremstyle{plain}
\newtheorem{thm}{\bf Theorem}[section]
\newtheorem{prop}[thm]{\bf Proposition}
\newtheorem{lem}[thm]{\bf Lemma}
\newtheorem{cor}[thm]{\bf Corollary}

\theoremstyle{definition}
\newtheorem{dfn}[thm]{\bf Definition}

\newtheorem{ex}[thm]{\bf Example}

\theoremstyle{remark}
\newtheorem{rem}[thm]{\bf Remark}

\def \A{\mathbb{A}}
\def \G{\mathbb{G}_m}

\def \Z{\mathbb{Z}}
\def \Q{\mathbb{Q}}
\def \C{\mathbb{C}}

\def \P{\mathbb{P}}

\def \O{\mathcal{O}}
\def \K{{\rm op}K_T}
\def \KG{{\rm op}K_G}
\def \KTT{{\rm op}K_{T\times T}}

\newcommand{\acknowledge}{\subsection*{Acknowledgments}}
\newcommand{\conventions}{\subsection*{Conventions}}
\newcommand{\notation}{\subsection*{Notation}}

\title[Localization in equivariant operational $K$-theory]
{Localization in equivariant operational $K$-theory and the 
Chang-Skjelbred property}
\author[Richard P. Gonzales]{Richard P. Gonzales\,*}\thanks{* 
Supported by 
the Max-Planck-Institut f\"{u}r Mathematik and  
the Institut des Hautes \'{E}tudes Scientifiques} 
\address{
Max-Planck-Institut f\"{u}r Mathematik\\
Vivatsgasse 7 \\
53111 Bonn\\
Germany
}
\address{
Institut des Hautes \'{E}tudes Scientifiques\\ 
35 Route de Chartres\\
F-91440 Bures-sur-Yvette\\
France
\vspace{.5cm}
}
\email{rgonzalesv@ihes.fr}
\begin{document}

\begin{abstract}
We establish a localization theorem of Borel-Atiyah-Segal type for 
the equivariant 
operational $K$-theory of Anderson and Payne \cite{ap:opk}. 
Inspired by the work of Chang-Skjelbred 
and Goresky-Kottwitz-MacPherson, 
we establish a general form of GKM theory in this setting,  
applicable to 
singular schemes with torus action.  
Our results are deduced from 
those in the smooth case 
via 
Gillet-Kimura's  
technique of cohomological descent for equivariant 
envelopes.        
As an application,   
we extend Uma's description of the equivariant 
$K$-theory of smooth compactifications of 
reductive groups 
to the equivariant operational $K$-theory of all, 
possibly singular,  
projective group embeddings. 
%
\end{abstract}

\maketitle


\section{Introduction and Motivation}

Goresky, Kottwitz and MacPherson, in their seminal paper 
\cite{gkm:eqc},  
developed a theory, nowadays called GKM theory, that makes it possible 
to describe the equivariant cohomology of certain {\em $T$-skeletal} varieties: 
complete algebraic varieties upon which a complex 
algebraic torus $T$ acts
with a finite number of fixed points and invariant curves. 
Let $X$ be a $T$-skeletal variety and denote by $X^T$ the fixed point set. 
The main purpose of GKM theory is to identify 
the image of the functorial map 
$
i^* : H^*_T(X) \to H^*_T(X^T),
$
assuming $X$ is {\em equivariantly formal}. 
GKM theory has been mostly applied to smooth projective $T$-skeletal varieties, 
because of the 
Bialynicki-Birula decomposition \cite{bb:torus}.
Additionally, the GKM data
issued from the fixed points and invariant curves 
has been explicitly obtained
for some interesting cases: 
flag varieties \cite{c:schu}, 
and regular embeddings of reductive groups 
\cite{bri:eqchow, bri:bru}. 
In contrast, regarding singular varieties, GKM theory has been applied
to Schubert varieties \cite{c:schu} and  
to {\em rationally smooth} projective group embeddings, 
due to the author's work \cite{go:cells, go:equiv}. 

\smallskip 

Because of its power as a computational tool, 
GKM theory has been 
implemented in other equivariant cohomology  
theories on schemes with torus actions.     
For instance, Brion established GKM theory 
for equivariant Chow groups \cite{bri:eqchow}, 
Vezzosi-Vistoli did it for equivariant algebraic $K$-theory \cite{vv:hkth}, 
and Krishna  provided the tool in equivariant algebraic cobordism \cite{kri:cob}. 
Nevertheless, in all of these generalizations, 
a 
crucial assumption on  
{\em smoothness} of the ambient space needs to be made.  

\smallskip

This paper is concerned with 
%
the equivariant $K$-theory of possibly 
{\em singular} 
schemes equipped with an action of an algebraic torus 
$T$ (i.e. $T$-schemes). 
Our main goal is to increase 
the applicability of
GKM theory as a tool 
for understanding the geometry of singular $T$-schemes in this setting.    
%
%
For convenience of the reader, 
we 
briefly review 
some of the basic underlying notions, as well as 
the previous progress made on this problem.  
Equivariant $K$-theory was 
developed by Thomason \cite{th:alg}. 
%
%
Let $X$ be a $T$-scheme. 
Let $K_T(X)$ denote the Grothendieck group of $T$-equivariant 
vector bundles on $X$. This is a ring, with the product given by the tensor product of equivariant vector bundles.  Let $K^T(X)$ denote the Grothendieck group of $T$-equivariant coherent sheaves on $X$.  
This is a module for the ring $K_T(X)$.  If we identify the representation ring $R(T)$ with $K_T(pt)$, then pullback by the projection $X\to pt$ gives a
natural map $R(T)\to K_T(X)$. In this way, $K_T(X)$ becomes an $R(T)$-algebra
and $K^T(X)$ an $R(T)$-module. 
The functor $K_T(-)$ is contravariant with respect to arbitrary equivariant maps. In contrast, $K^T(-)$ is covariant for equivariant proper morphisms and contravariant for
 equivariant flat maps. 
If $X$ is smooth, then every $T$-equivariant coherent
sheaf has a finite resolution by $T$-equivariant locally free sheaves, and thus 
$K_T(X)\simeq K^T(X)$.  
 When $X$ is complete, 
the equivariant Euler characteristic 
$$\mathcal{F}\mapsto \chi(X,\mathcal{F})=\sum_i (-1)^i [H^i(X,\mathcal{F})]$$
yields the pushforward map 
$\chi:K^T(X)\longrightarrow K^T(pt)\simeq R(T)$. 
%
%
%
%
By work of Merkurjev \cite{mer:comp}, one 
recovers the usual $K$-theory from the 
equivariant one via 
the identity 
$K^T(X)\otimes_{R(T)}\Z \simeq K(X)$. 


\medskip

In general, 
the $K$-theory groups are difficult to compute. 
In the case of singular varieties, 
they can be quite large \cite[Introduction, p. 2]{ap:opk}. 
In the smooth case, however, there are three powerful theorems that allow many computations and important comparison theorems of Riemann-Roch type. 
The first one is the localization theorem of Borel-Atiyah-Segal type.  

\medskip

\noindent{\bf Localization theorem of Borel-Atiyah-Segal type} 
(\cite[Th\'eor\`eme 2.1]{th:con}){\bf .}
{\em Let $X$ be a smooth complete scheme with an action of $T$. 
Let $X^T$ be the subscheme of fixed points and let $i_T:X^T\to X$ be 
the natural inclusion. Then the pullback  
$i_T^*:K_T(X)\to K_T(X^T)$ 
is injective, and it becomes surjective over the quotient field of $R(T)$.} 

\medskip 

Let $X$ be a smooth complete $T$-scheme. 
The second fundamental theorem in this context identifies the image of $i^*_T$ inside 
$K_T(X^T)\simeq K(X^T)\otimes R(T)$. 
To state it, we introduce some notation. 
Let $H\subset T$ be a subtorus of codimension one. 
Observe that $i_T$ factors as $i_{T,H}:X^T\to X^H$ followed by 
$i_H:X^H \to X$. Thus, the image of $i^*_T$ is contained in the image of $i^*_{T,H}$. In symbols, 
$${\rm Im}[i^*_T:K^0_T(X)\to K^0_T(X^T)]\subseteqq \bigcap_{H\subset T} {\rm Im}[
i^*_{T,H}:K^0_T(X^{H})\to K^0_T(X^T)
],$$
where the intersection runs over all codimension-one subtori $H$ of $T$. 
This 
criteria, 
which 
dates back to the work of Chang-Skjelbred \cite{cs:schur} in equivariant cohomology, 
yields a complete description 
of 
$K_T(X)$ as a subring of $K_T(X^T)\simeq K(X^T)\otimes R(T)$. 
%
\medskip

\noindent{\bf CS property} (\cite[Theorem 2]{vv:hkth}){\bf .}   {\em Let $X$ be a smooth complete $T$-scheme. 
Then the image of the injective map 
$i^*_T:K_T(X)\to K_T(X^T)$ equals 
the intersection of the images of 
$i^*_{T,H}:K_T(X^H)\to K_T(X^T),$ 
where $H$ runs over all subtori of codimension one in $T$.} 

\medskip

Now let $X$ be a (complete) $T$-skeletal variety. 
Assume, for simplicity, that each $T$-invariant irreducible curve has exactly two fixed points (e.g. $X$ is equivariantly embedded in a normal $T$-variety). 
In this setting, it is possible to define a ring $PE_T(X)$ of 
{\em piecewise exponential} functions. 
Indeed, let $K_T(X^T)=\oplus_{x\in X^T}R_x$, where $R_x$ is a copy of the representation ring $R(T)$. We then define $PE_T(X)$ as the subalgebra of $K_T(X^T)$ given by 
$$
PE_T(X)=\{(f_1,\ldots, f_m)\in \oplus_{x\in X^T}R_x\,|\, f_i\equiv f_j \mod 1-e^{-\chi_{i,j}}\}
$$ 
where $x_i$ and $x_j$ are the two distinct fixed points in the closure of the one-dimensional $T$-orbit $C_{i,j}$, and $\chi_{i,j}$ is the character of $T$ associated with $C_{i,j}$. This character is uniquely determined up to sign (permuting the two fixed points changes $\chi_{i,j}$ to its opposite). 
In light of the CS property, one obtains:

\medskip

\noindent {\bf GKM theorem} (\cite[Corollary 5.12]{vv:hkth}, \cite[Theorem 1.3]{uma:kth}){\bf .} 
{\em Let $X$ be a smooth $T$-skeletal variety. Then $i^*_T:K_T(X)\to K_T(X^T)$ induces an isomorphism between $K_T(X)$ and $PE_T(X)$. If $X$ is also projective, 
then $K_T(X)$ is a free $R(T)$-module of rank $|X^T|$. }

\medskip

Thus far, it is clear that to any complete $T$-skeletal
variety $X$ we can associate the ring $PE_T(X)$,  
regardless of whether $X$ is smooth or not. 
(In fact, if $X$ is a projective 
compactification of a reductive group $G$ with maximal torus $T$, 
then $X$ is $T\times T$-skeletal, and $PE_{T\times T}(X)$ 
has been explicitly identified in \cite{go:equiv}.)  
Nonetheless, as it stems from the previous facts, 
$PE_T(X)$ does not always describe $K_T(X)$.  
This phenomena yields some natural 
{\em questions:} Let $X$ be a $T$-skeletal variety. 
What kind of information does $PE_T(X)$ encode? If not 
equivariant $K$-theory, is it still reasonable to expect that $PE_T(X)$ encodes certain topological/geometric information that is {\em common} to all possible $T$-equivariant resolution of singularities of $X$? The work of Payne \cite{p:t} and Anderson-Payne \cite{ap:opk}, inspired in turn by the works of Fulton-MacPherson-Sottile-Sturmfels \cite{f:sph} and Totaro \cite{to:linear}, gives a positive answer to these questions 
 {\em when $X$ is a toric variety}. Namely, the GKM data (i.e. $PE_T(X)$) of a toric variety encodes all the information needed to reconstruct Bott-Chern operators defined on the structure sheaves $\mathcal{O}_{\overline{Tx}}$ of the $T$-orbit closures $\overline{Tx}\subseteq X$ (and their equivariant resolutions).    
This positive result is our motivation. 
In the pages to follow we will show that Anderson-Payne's 
assertion on toric varieties holds more generally 
for {\em all} $T$-skeletal varieties. 
But first, and in order to  
put these statements in a much clearer form,  
we recall some of the main aspects of Anderson and Payne's 
equivariant operational $K$-theory.

%
%
\medskip

Fulton and MacPherson \cite{fm:op} devised a 
machinery that 
produces a cohomology theory out of a homology theory. This cohomology has all the formal properties one could hope for, 
and it is well suited for the study of singular schemes. 
Taking as input the homology functor $K^T(-)$, 
Anderson-Payne \cite{ap:opk} 
obtained a theory that is very well suited for computations. 
Moreover, it agrees with Thomason's equivariant $K$-theory 
when $X$ is smooth (Properties (a) and (b) below).   
We outline here the main notions of \cite{ap:opk}.  
Let $X$ be a $T$-scheme. The 
{\em $T$-equivariant operational $K$-theory} ring of $X$, denoted $\K(X)$, is defined as follows: 
an element $c\in \K(X)$ is a collection 
of homomorphisms $c_f:K^T(Y)\to K^T(Y)$  
for every $T$-map $f:Y\to X$. (Recall that $K^T(Y)$ denotes the  
Grothendieck group of $T$-equivariant coherent sheaves on $Y$.)   
These homomorphisms 
must be compatible with ($T$-equivariant) proper pushforward, flat pullback and Gysin morphisms \cite{ap:opk}. 
For any $X$, the ring structure on $\K(X)$ 
is given by composition of 
such homomorphisms. With this product, $\K(X)$ 
becomes an associative commutative ring with unit.  
Moreover, $\K(X)$ is contravariantly functorial in $X$. 
Other salient functorial properties of $\K(-)$ are:

\begin{enumerate}[(a)]
%
%
%
 \item For any $X$, there is a canonical homomorphism 
$K_T(X)\to \K(X)$ of $R(T)$-algebras, 
sending a class $\gamma$ to the operator $[\gamma]$ which acts via $[\gamma]_g=g^*\gamma\cdot \xi$, for any $T$-map $g:Y\to X$ and $\xi\in K^T(Y)$. There is also a canonical map $\K(X)\to K^T(X)$ defined by $c\mapsto c_{{\rm id}_X}[\mathcal{O}_X]$, where $\mathcal{O}_X$ is the structure sheaf of $X$. Put together, they provide a factorization of the canonical homomorphism 
$K_T(X)\to K^T(X)$ \cite[Theorem 5.6]{ap:opk}.    

\smallskip

\item When $X$ is smooth, the homomorphisms 
$$K_T(X)\to \K(X)\to K^T(X),$$ 
defined in (a), 
are all isomorphisms 
of $R(T)$-modules \cite[Corollary 4.5 and Theorem 5.6]{ap:opk}.

\smallskip

\item {\em $\A^1$-homotopy invariance \cite[Corollary 4.7]{ap:opk}:} 
For any scheme $X$, the natural pull back map from 
$\K(X)$ to $\K(X\times \A^1)$ is an isomorphism.  

\smallskip

\item {\em Gillet-Kimura's cohomological descent for equivariant 
envelopes \cite[Theorem 5.3]{ap:opk}:}  \label{gillet1.thm}
If $\pi:\tilde{X}\to X$ is an equivariant envelope (that is, 
any $T$-invariant subvariety of $X$ is the birational 
image 
of a $T$-invariant subvariety of $\tilde{X}$)  
and $\pi_1$, $\pi_2$ are the projections 
$\tilde{X}\times_X \tilde{X}\to \tilde{X}$, 
then the following sequence is exact 
$$
\xymatrix{
0 \ar[r]& \K(X) \ar[r]^{\pi^*}& 
\K(\tilde{X}) \ar[rr]^{\pi_1^*-\;\pi_2^*\;\;\;\;\;}& &\K(\tilde{X}\times_X \tilde{X}).
}
$$
%
\smallskip

\item Let $X$ be a complete $T$-variety. 
If $X$ is a toric variety (i.e. $X$ is normal and 
has a dense orbit 
isomorphic to $T$), then $\K(X)\simeq PE_T(X)$ 
\cite[Theorem 1.6]{ap:opk}. Similar results hold for 
non-complete toric varieties \cite{ap:opk}.    

\smallskip

\item {\em Equivariant Kronecker duality 
for spherical varieties \cite[Theorem 6.1]{ap:opk}:}  
Let $B$ be a connected solvable linear algebraic group 
with maximal torus $T$. Let $X$ be a scheme with an 
action of $B$. 
If $B$ acts on $X$ with finitely many orbits, then 
the natural equivariant 
Kronecker map 
$$\mathcal{K}_T:\K(X) \to {\rm Hom}_{R(T)}(K^T(X),R(T)),$$  
induced by pushforward to a point, namely,  
$$\mathcal{K}_T:\; c\longmapsto \{\xi\mapsto \chi(X,c_{{\rm id}_X}(\xi))\}_{\,,}$$
is an isomorphism. This holds e.g. for Schubert varieties and 
spherical varieties.   
There is a more general version of 
equivariant Kronecker duality, valid for   
$T$-linear schemes (\cite[Section 6]{ap:opk}). This  
class encompasses all  
the $B$-schemes mentioned above   
(see e.g. \cite[Theorem 2.5]{go:tlinear}).  
For equivariant Kronecker 
duality in the context of 
equivariant operational Chow groups, see 
\cite[Theorem 3.6]{go:tlinear}.  
\end{enumerate}

\medskip

In this paper we use the functorial 
properties listed above, together with resolution of singularities, 
to establish: 
\begin{enumerate}[(I)]
 \item The localization theorem of Borel-Atiyah-Segal type for $\K(X)$, 
whenever $X$ is a complete $T$-scheme 
(Theorem \ref{localization.thm}). 
 \smallskip

 \item The CS property for $\K(X)$, 
where $X$ is any complete $T$-scheme (Theorem \ref{cs.thm}). 
\smallskip

\item GKM theory for possibly singular 
complete 
$T$-varieties: if $X$ is a $T$-skeletal variety, then 
$\K(X)\simeq PE_T(X)$ (Theorem \ref{gkm.thm}). 
\end{enumerate}
Together with the combinatorial results of \cite{go:equiv}, this 
extends Anderson's and Payne's work on 
toric varieties to all 
projective group embeddings 
(Theorems \ref{gkm.emb.thm} and \ref{comparison.emb.thm}). 
See Section 7 (as well as \cite{go:tlinear})   
for the corresponding statements in operational Chow groups with 
rational coefficients.

\acknowledge{
The research in this paper was done 
during my visit to the 
Max-Planck-Institut f\"{u}r Mathematik (MPIM) and the 
Institute
des Hautes \'Etudes Scientifiques (IHES).  
I am deeply grateful to both institutions for their support,  
outstanding hospitality, and excellent working conditions.
}





\section{Conventions and Notation}

\conventions{
Throughout this paper, we fix an algebraically closed 
field $k$ of characteristic zero. 
All schemes and algebraic groups are assumed to be defined 
over $k$. By a scheme we mean a separated scheme 
of finite type. 
A variety is a reduced scheme. 
Observe that varieties need not be irreducible. 
A subvariety is a closed subscheme which is a variety. 
A curve on a scheme is an irreducible one-dimensional subscheme.    
Unless explicit mention is made to the contrary, 
we will assume all schemes are equidimensional. 
A point on a scheme will always be a closed point.
}

\smallskip

\notation{
We denote by $T$ an algebraic torus. 
A scheme $X$ provided with an algebraic action of $T$ is called a 
{\em $T$-scheme}. 
If $X$ is a $T$-scheme, the class in $K^T(X)$ of a $T$-equivariant 
coherent sheaf $\mathcal{F}$ will be denoted by $[\mathcal{F}]$. 
In particular, if $Y\subset X$ is a $T$-stable closed subscheme, 
then the structure sheaf of $Y$ defines a class $[\mathcal{O}_Y]$ 
in $K^T(X)$. For a $T$-scheme $X$,  
we denote by $X^T$ the fixed point subscheme and by $i_T:X^T\to X$
the natural inclusion. 
If $H$ is a closed subgroup of $T$, we similarly denote by 
$i_H:X^H\to X$ the inclusion of the fixed point subscheme. 
When comparing $X^T$ and $X^H$ we write $i_{T,H}:X^T\to X^H$ 
for the natural ($T$-equivariant) inclusion. 
If $g:Y\to X$ is a $T$-equivariant 
morphism of $T$-schemes, 
then we write   
$g_T:Y^T\to X^T$, or simply $g:Y^T\to X^T$,  
for the 
associated morphism of fixed point subschemes.  
Likewise, we write  
$g^*:\K(X)\to \K(Y)$ for the  
pullback in equivariant operational 
$K$-theory. 
%

\smallskip

We denote by $\Delta$ the character group of $T$, and by 
$\Z[\Delta]$ the group ring over $\Z$ of $\Delta$.  
We let $e^{\chi}$ denote the element of $\Z[\Delta]$ 
corresponding to $\chi\in \Delta$. Then $\{e^\chi\}_{\chi \in \Delta}$ 
is a basis of the $\Z$-module $\Z[\Delta]$. 
For a $k$-linear 
representation $V$ of $T$, we put 
$${\rm tr}(V)=\sum_{\chi \in \Delta}({\rm rank}_kV_{\chi})e^{\chi},$$
where $V_\chi$ is the subspace of invariants of $T$ of weight $\chi$ in $V$. 
It is well-known that ${\rm tr}$ induces an isomorphism 
from the representation ring of $T$, denoted $R(T)$, to $\Z[\Delta]$.   
}


\section{Equivariant envelopes and computability of 
equivariant operational $K$-theory}
Recall that an envelope 
$p:\tilde{X}\to X$ is a proper map such that for any 
subvariety $W\subset X$ there is a subvariety $\tilde{W}$
mapping birationally to $W$ via $p$ (\cite[Definition 18.3]{f:int}). 
In the case of $T$-actions, 
we say that $p:\tilde{X}\to X$ is an {\em equivariant envelope} 
if $p$ is $T$-equivariant, and if we can take $\tilde{W}$ to be 
$T$-invariant for $T$-invariant $W$. If there is an open set 
$U\subset X$ over which $p$ is an isomorphism, then we say that 
$p:\tilde{X}\to X$ is a {\em birational} envelope. 
The following is recorded in \cite[Proposition 7.5]{eg:cycles}.

\begin{lem}\label{smooth_envelope.lem}
Let $X$ be a $T$-scheme. Then there exists a 
$T$-equivariant birational envelope $p:\tilde{X}\to X$, 
where $\tilde{X}$ is a smooth quasi-projective $T$-scheme. \hfill $\square$      
\end{lem}



Anderson and Payne's version of Gillet and Kimura's notion of 
cohomological descent (Property (\ref{gillet1.thm}), Introduction)
implies that $\K(X)$ of a singular scheme $X$ 
injects into $\K(\tilde{X})$ of a smooth equivariant envelope (which is 
the usual equivariant $K$-theory ring of a smooth scheme) with an explicit cokernel. 
More precisely, suppose that $p:\tilde{X}\to X$ is a $T$-equivariant birational 
envelope which is an 
isomorphism over an open set $U\subset X$. Let $\{Z_i\}$ be the irreducible 
components of $Z=X-U$, and let $E_i=p^{-1}(Z_i)$, with $p_i:E_i\to Z_i$ 
denoting the restriction of $p$. 
The next theorem is Kimura's fundamental result  
\cite[Theorem 3.1]{ki:op}) adapted 
to equivariant operational $K$-theory.



\begin{thm}[\protect{\cite[Theorem 5.4]{ap:opk}}] \label{gillet.thm}
Let $p:\tilde{X}\to X$ be a $T$-equivariant envelope. 
Then the induced map $p^*:\K(X)\to \K(\tilde{X})$ is injective.
Furthermore, if $p$ is birational (and notation is as above), 
then  
the image of $p^*$ is described inductively 
as follows: a class $\tilde{c}\in \K(\tilde{X})$ equals $p^*(c)$, for some 
$c\in \K(X)$ if and only if, for all $i$, we have $\tilde{c}|_{E_i}=p^*_i(c_i)$ for 
some $c_i\in \K(Z_i)$.  
\hfill $\square$ 
\end{thm}

Since $E_i$ and $Z_i$ may be arranged to 
have smaller dimension than that of $X$, we can use 
this result to
compute $\K(X)$ using a resolution of singularities 
(Lemma \ref{smooth_envelope.lem}) and induction on dimension.  
This is one of the reasons 
why 
cohomological descent 
(Property (\ref{gillet1.thm}), Introduction)   
makes equivariant operational $K$-theory more computable than the 
Grothendieck ring of equivariant vector bundles, when it comes to  
singular $T$-schemes.  



\begin{cor}\label{ez.seq.cor}
Notation being as above,  the sequence 
$$
\xymatrix{
0 \ar[r]& \K(X) \ar[r]& 
\K(\tilde{X})\oplus \K(Z) \ar[r]& 
\K(E)
}
$$  
is exact, where $E=p^{-1}(Z)$. \hfill $\square$
\end{cor}



\begin{cor}\label{irr.seq.cor}
Let $Y$ be a $T$-scheme, and let  
$Y=\cup_{i=1}^n Y_i$
%
be the decomposition of $Y$ into 
irreducible components. 
Let $Y_{ij}=Y_i\cap Y_j$. 
Then the sequence 
$$
0\to \K(Y)\to \bigoplus_i\K(Y_i)\to \bigoplus_{i,j}\K(Y_{ij}).
$$  
is exact. 
\end{cor}

\begin{proof}
First 
recall that 
$\bigsqcup_i Y_i\to Y$ is an equivariant envelope.
Now use cohomological descent 
to get the result.   
\end{proof}

The following was first observed in \cite[Lemma 7.2]{eg:cycles}. 

\begin{lem}\label{envelope.fix.subgp.lem}
Let $X$ be a 
$T$-scheme, and let 
$\pi:\tilde{X}\to X$ be an equivariant envelope.
If $H$ is a closed subgroup of $T$, then  
the induced map 
$\tilde{X}^H\to X^H$ 
is also a $T$-equivariant envelope.
\end{lem}

\begin{proof} The argument  here is basically that of \cite[Lemma 7.2]{eg:cycles}.  
First, notice that the map $\pi_H:\tilde{X}^H\to X^H$ is $T$-equivariant, because 
$T$ is an abelian group. 
Now let $W\subset X^H$ be a $T$-invariant irreducible subvariety 
and let $\tilde{W}$ be an irreducible subvariety of 
$\tilde{X}$ mapping birationally to $W$ via $\pi$. 
To prove that $\tilde{X}^H$ is an equivariant envelope, 
it suffices to prove that we can take $\tilde{W}\subset \tilde{X}^H$. 
The restricted map $\pi: \tilde{W}\to W$ is a $T$-equivariant isomorphism 
over a dense open subspace $U$ of $W$. 
Replace $\tilde{W}$ with the closure of 
$\pi^{-1}(U)$. Because $H$ acts trivially 
on $\pi^{-1}(U)$ (for $U\subset W\subset  X^H$), 
and $\tilde{W}^H$ is closed, we get $\tilde{W}\subset \tilde{X}^H$, 
as desired.
\end{proof}

An important technical result is stated next. 

\begin{cor}\label{gillet.square.cor}
Let $p:\tilde{X}\to X$ be an equivariant envelope. 
If $H$ is a closed subgroup of $T$, then the diagram of exact sequences 
$$
\xymatrix{
0\ar[r]& \K(X) \ar[r]\ar[d]^{i^*_H}& 
\K(\tilde{X})\ar[r]\ar[d]\ar[d]^{i^*_H}& 
\K(\tilde{X}\times_X \tilde{X})\ar[d]\ar[d]^{i^*_H} \\
0\ar[r]& 
\K(X^H) \ar[r]& 
\K(\tilde{X}^H)\ar[r]& 
\K(\tilde{X}^H\times_{X^H} \tilde{X}^H).
}
$$
commutes. 
Moreover, if $p$ is birational, and notation being as in Theorem 
\ref{gillet.thm}, then the diagram of exact sequences  
$$
\xymatrix{
0\ar[r]& \K(X) \ar[r]\ar[d]^{i^*_H}& 
\K(\tilde{X})\oplus \K(Z)\ar[r]\ar[d]\ar[d]^{i^*_H}& 
\K(E)\ar[d]\ar[d]^{i^*_H} \\
0\ar[r]& 
\K(X^H) \ar[r]& 
\K(\tilde{X}^H)\oplus \K(Z^H)\ar[r]& 
\K(E^H).
}
$$
commutes.  
\end{cor}

\begin{proof}
First, apply cohomological descent to $p:\tilde{X}\to X$. 
Due to Lemma \ref{envelope.fix.subgp.lem}, we can also 
apply this tool to the $T$-equivariant 
envelope $p_H:\tilde{X}^H\to X^H$, 
noticing that $(\tilde{X}\times_X \tilde{X})^H
=\tilde{X}^H\times_{X^H}\tilde{X}^H$. 
Now write the associated short exact sequences 
as the rows of the first square diagram displayed above. 
An straightforward check shows that 
the diagram is commutative. A similar argument 
yields the second assertion, 
in view of 
Corollary \ref{ez.seq.cor}.  
\end{proof}

\smallskip

In the upcoming proposition 
we state another crucial consequence of Kimura's work. 
Put in perspective, it asserts that the 
equivariant operational $K$-theory ring $\K(X)$ of 
{\em any} complete $T$-scheme $X$ is a subring of $\K(X^T)$. 
Moreover, 
there is a natural isomorphism 
$$\K(X^T)\simeq \K(X^T)\otimes_{\Z} R(T),$$ 
by \cite[Corollary 5.5]{ap:opk}.  
In many cases of interest, $X^T$ is finite 
(e.g. for spherical varieties) 
and so one has $\K(X)\subseteq \bigoplus^\ell_1 \K(pt)=R(T)^{\ell}$, 
where $\ell=|X^T|$.  
This motivates our introduction of localization techniques, and ultimately 
GKM theory, into the study of 
the functor $\K(-)$. 

\begin{prop}\label{inj.fix.set.lem}
Let $X$ be a complete 
$T$-scheme and let 
$i_T:X^T\to X$ be the inclusion of 
the fixed point subscheme. Then the pull-back map 
$$i^*_T:\K(X)\to \K(X^T)$$
is injective.  
\end{prop}

\begin{proof}
First, choose 
a $T$-equivariant envelope 
$p:\tilde{X}\to X$, with $\tilde{X}$ projective and smooth 
(Lemma \ref{smooth_envelope.lem}).   
Thus 
$p^*:\K(X)\to \K(\tilde{X})$ 
is injective (Theorem \ref{gillet.thm}). Since 
$\tilde{X}$ is smooth and projective, 
the pull-back 
$i^*_T:\K(\tilde{X})\to \K({\tilde{X}}^T)$, 
is injective  
(by Property (b) of the Introduction together with 
the CS property for smooth $T$-schemes \cite[Theorem 2]{vv:hkth}).  
Besides, the chain of inclusions 
$\tilde{X}^{T}\subset p^{-1}(X^T)\subset \tilde{X}$  
indicate that $\tilde{i}^*_T$ factors through 
$\iota^*:\K(\tilde{X})\to \K(p^{-1}(X^T))$, 
where 
$\iota:p^{-1}(X^T)\hookrightarrow \tilde{X}$ 
is the natural inclusion. 
Thus, $\iota^*$
is injective as well. 
%
Finally, adding this information to the commutative diagram 
$$
\xymatrix{
\K(X)\ar[r]^{p^*} \ar[d]^{i_T^*}& \K(\tilde{X})\ar[d]^{\iota^*}\\
\K(X^T) \ar[r]^{p^*}& \K(p^{-1}(X^T)). 
}
$$
renders $i^*_T:\K(X)\to \K(X^T)$ injective. 
\end{proof}

\begin{cor}\label{injectivity.cor}
Let $X$ be a complete $T$-scheme. Let $Y$ be a $T$-invariant 
closed subscheme containing $X^T$. Denote by $\iota:Y\to X$ the 
natural inclusion. 
Then the $R(T)$-algebra map $\iota^*:\K(X)\to \K(Y)$ is injective.
In particular, if $H$ is a closed subgroup of $T$,  
then 
$i_{H}^*:\K(X)\to \K(X^H)$ is injective. 
\end{cor}

\begin{proof}
Simply notice that $\iota:Y\to X$ fits into the commutative 
triangle
$$
\xymatrix@R=1em{
                                  & Y  \ar[rd]^{\iota}&    \\
X^T \ar[rr]_{i_T}\ar[ru]^{i_{T,Y}}  &               & X.    \\ 
}
$$
In other words, the functorial map $i^*_T:\K(X)\to \K(X^T)$ factors 
as $\iota^*:\K(X)\to \K(Y)$ followed by $i^*_{T,Y}: \K(Y)\to \K(X^T)$. 
By Proposition \ref{inj.fix.set.lem}, 
$i^*_T$ is injective,  hence so is $\iota^*$. 
As for the second 
assertion, just note that $X^H$ is $T$-invariant and $X^T\subset X^H$.   
\end{proof}


\begin{rem}\label{normalization.irr.comp.rem}
Of particular interest is the case $Y=\cup_{i=1}^n Y_i$, where $Y_i$ are 
the irreducible components of $Y$. Let $Y_{ij}=Y_i\cap Y_j$. 
By Corollary \ref{irr.seq.cor} the following sequence is exact 
$$
0\to \K(Y)\to \bigoplus_i \K(Y_i)\to \bigoplus_{i,j}\K(Y_{ij}).
$$ 
When $Y^T$ is finite, the sequence above yields the 
commutative diagram (Corollary \ref{gillet.square.cor}):
$$
\xymatrix{
0\ar[r]& \K(Y)\ar[r]\ar[d]^{i_T^*}& \bigoplus_i \K(Y_i)\ar[r]\ar[d]^{i_T^*}& \bigoplus_{i,j}\K(Y_{ij})\ar[d]^{i_T^*} \\
0\ar[r]& \K(Y^T)\ar[r]^{p}& \bigoplus_i \K(Y^T_i)\ar[r]^{q}& \bigoplus_{i,j}\K(Y^T_{ij})
}
$$
Since all vertical maps are injective (Proposition \ref{inj.fix.set.lem}),  
it is important to observe 
that we can describe the image of the first 
vertical map in terms of the image of the second vertical map 
and the kernel of $q$. In other words, 
the map 
$$p:im(i^*_{T,Y})\to \{w\in \bigoplus_i \K(Y_i)\,|\, w\in im(\sum_i i^*_T)\,{\rm and}\,q(w)=0\}$$
sending $u\to p(u)$ is an isomorphism. Now, since $Y^T$ is finite, 
the kernel of the map $q$ consists of all families $(f_i)_i$ such that 
$f_i(x_k)=f_j(x_k)$ (equality of $k$-components), whenever $x_k$ is in the intersection of $Y_i$ and $Y_j$.
\end{rem}

\medskip

Back to the general case, let $X$ be a complete
$T$-scheme. 
We wish to describe the image of the injective 
map $$i^*_T:\K(X)\to \K(X^T).$$
For this, let $T'\subset T$ be a subtorus of codimension one. Observe that 
$i_T:X^T\to X$ factors as $i_{T,T'}:X^T\to X^{T'}$ followed 
by $i_{T'}:X^{T'}\to X^T$. Thus, the image of $i^*_T$ is contained in the image 
of $i^*_{T,T'}$. In symbols, 
$${\rm Im}[i^*_T:\K(X)\to \K(X^T)]\subseteqq \bigcap_{T'\subset T} {\rm Im}[
i^*_{T,T'}:\K(X^{T'})\to \K(X^T)
],$$
where the intersection runs over all codimension-one subtori of $T$. 
This observation will lead, as in the smooth case, 
to an explicit description of the image of $i_T^*$. 
Such is the central theme of the subsequent sections.

\section{The Localization theorem 
of Borel-Atiyah-Segal type and the Chang-Skjelbred property}

Let $T$ be an algebraic torus. 
We recall a construction of Thomason 
\cite[Lemma 1.1, Proposition 1.2]{th:con}.
Let $\mathfrak{p}\subset R(T)$ be a 
prime ideal. Set 
$K_\mathfrak{p}=\{n\in \Delta \;|\; 1-n \in \mathfrak{p}\}$, 
where $\Delta$ is the character group of $T$.  
It is well-known, see e.g. \cite{bo:lag}, that 
the quotient $\Delta/K_{\mathfrak{p}}$ determines a 
unique subgroup $T_{\mathfrak{p}}\subset T$ with the property that 
$R(T_{\mathfrak{p}})=\Z[\Delta/K_{\mathfrak{p}}]$. 
Following \cite{th:con}, we 
call $T_{\mathfrak{p}}$ the {\em support} of $\mathfrak{p}$. 
When $\mathfrak{p}$ is maximal, $K_{\mathfrak{p}}$ 
has finite index and $T_{\mathfrak{p}}$ is a finite group.

\begin{thm}\label{localization.thm}
Let $X$ be a $T$-scheme. Let $\mathfrak{p}\subset R(T)$ be 
a prime ideal and $T_{\mathfrak{p}}$ be its support. 
Then the $R(T)$-algebra map 
$i_{T_\mathfrak{p}}^*:\K(X)\to \K(X^{T_{\mathfrak{p}}})$ 
becomes an isomorphism after localizing at $\mathfrak{p}$: 
$$\xymatrix{i^*_{T_\mathfrak{p}}:\K(X)_{\mathfrak{p}}\ar[r]^{\sim}&\K(X^{T_{\mathfrak{p}}})_{\mathfrak{p}}}.$$  
\end{thm}

\begin{proof}
Choose a $T$-equivariant birational envelope $p:\tilde{X}\to X$, 
with $\tilde{X}$ quasi-projective and smooth. Then $p$ is an 
isomorphism outside some $T$-invariant closed subscheme $Z$.
Let $E=p^{-1}(Z)$. Notice that $p$ can be chosen so 
that $Z$ and $E$ have dimension smaller than that of $X$.
Now, in light of Corollary \ref{gillet.square.cor}, form the 
commutative diagram of exact sequences  
$$
\xymatrix{
0\ar[r]& \K(X) \ar[r]\ar[d]^{i^*_{T_\mathfrak{p}}}& \K(\tilde{X})\oplus \K(Z)\ar[r]\ar[d]\ar[d]^{i^*_{T_\mathfrak{p}}}& \K(E)\ar[d]\ar[d]^{i^*_{T_\mathfrak{p}}} \\
0\ar[r]& \K(X^{T_\mathfrak{p}}) \ar[r]& \K(\tilde{X}^{T_\mathfrak{p}})\oplus \K(Z^{T_\mathfrak{p}})\ar[r]& \K(E^{T_\mathfrak{p}}).
}
$$ 
It follows from Noetherian induction and 
Thomason's concentration theorem 
\cite[Th\'eor\`eme 2.1]{th:con} that 
 the last two vertical maps become  
isomorphisms after localizing at $\mathfrak{p}$; hence 
so does the first one. 
\end{proof}

\begin{cor}
Let $X$ be a complete $T$-scheme. Notation being as above, the 
$R(T)$-algebra map $i^*_{T_\mathfrak{p}}:\K(X)\to \K(X^{T_{\mathfrak{p}}})$ is injective and it becomes surjective after localizing at $\mathfrak{p}$. 
\end{cor}

\begin{proof}
Use Proposition \ref{inj.fix.set.lem}
and Theorem \ref{localization.thm}. 
\end{proof}



\begin{dfn}
Let $X$ be a complete $T$-scheme. We say that $X$ has the 
{\bf Chang-Skjelbred property} (or {\bf CS property}, for short) if the 
image of $$i^*_T:\K(X)\to \K(X^T)$$ is exactly the intersection 
of the images of 
$$
i^*_{T,H}:\K(X^H)\to \K(X^T),
$$
where $H$ runs over all subtori of codimension one in $T$.
\end{dfn}

By \cite[Theorem 2]{vv:hkth}, every nonsingular complete $T$-scheme
has the CS property. Remarkably, it holds over $\Z$.  
We extend this result to include all, possibly singular, 
complete schemes with an action of $T$. 


\begin{thm}\label{cs.thm}
Let $X$ be a complete $T$-scheme. 
Then $X$ has the CS property.  
\end{thm}

\begin{proof}
Let $\pi:\tilde{X}\to X$ be a $T$-equivariant envelope 
with $\tilde{X}$ projective and smooth (Lemma \ref{smooth_envelope.lem}). 
Because of Corollary \ref{gillet.square.cor}   
we get the commutative diagram 
$$
\xymatrix{
0 \ar[r]& \K(X) \ar[r]^{\pi^*}\ar[d]_{i_T^*}& \K(\tilde{X}) \ar[rr]^{\pi_1^*-\pi_2^*} \ar[d]_{\widetilde{i_T^*}}& & \K(X')\ar[d]_{{i_T^*}'}\\
0 \ar[r]& \K(X^T) \ar[r]^{{\pi_T}^*}& \K({\tilde{X}}^T)\ar[rr]^{\pi_{T,1}^*-\pi_{T,2}^*} & & \K({X'}^T).\\
}
$$
A simple diagram chasing shows that $u\in \K(X^T)$ is in the image of $i^*_T$ 
if and only if 
${\pi_T}^*(u)$ is in the image of $\widetilde{i^*_T}$. Indeed, this follows from the fact 
that all vertical maps in the diagram are injective (Proposition \ref{inj.fix.set.lem}).  


On the other hand we have the commutative diagram  
$$
\xymatrix{
 & \K(X)   \ar[rr]^{\pi^*}\ar[d]^{i_T^*}\ar[ddl]_{i^*_H}& 
&\K(\tilde{X})\ar[d]^{\widetilde{i_T^*}}\ar[ddl]_{\widetilde{i^*_H}}\\
 & \K(X^T) \ar[rr]^{{\pi_T}^*}& &\K(\tilde{X}^T)\\
\K(X^H) \ar[rr]^{{\pi_H}^*}\ar[ur]_{i^*_{T,H}}& & 
\K(\tilde{X}^H)\ar[ur]_{\widetilde{i^*_{T,H}}} & \\
}
$$
obtained by combining 
and comparing  
the sequences that Corollary \ref{gillet.square.cor} assigns  
to the envelopes $\pi:\tilde{X}\to X$, $\pi_H:\tilde{X}^H\to X^H$ 
and $p_T:\tilde{X}^T\to X^T$. 
From the diagram it follows that if $u\in \K(X^T)$ is in 
the image of $i^*_{T,H}$, then ${\pi_T}^*(u)$ is in the image of $\widetilde{i^*_{T,H}}$. 
Hence, if $u$ is in the intersection of the images of all $i^*_{T,H}$, then 
${\pi_T}^*(u)$ is in the intersection of the images of all 
$\widetilde{i^*_{T,H}}$, where $H$ runs over all codimension-one subtori of $T$. 
Since $\tilde{X}$ satisfies the CS condition, then ${\pi_T}^*(u)$ is in the image of 
$\widetilde{i^*_T}$. Finally, 
from the observation made at the end of the previous paragraph, 
we conclude that $u$ is in the image of $i^*_T$. 
\end{proof}

\section{GKM theory} 

Vistoli and Vezzosi 
established a version of GKM theory   
applicable to nonsingular complete $T$-schemes 
\cite[Theorem 2]{vv:hkth}. 
Based on Theorem \ref{cs.thm},     
we establish here a version of 
GKM-theory valid for the equivariant operational 
$K$-theory of {\em singular} complete $T$-schemes 
(Theorem \ref{gkm.thm}). 
As a consequence, we extend  
\cite[Theorem 1.6]{ap:opk} to 
%
%
%
%
%
%
%
%
the larger class of $T$-skeletal varieties, a family 
of objects that includes all    
equivariant 
projective embeddings  
of reductive groups (Theorem \ref{gkm.emb.thm}). 
%
We start by recalling a few definitions from \cite{gkm:eqc} and \cite{go:cells}.

\begin{dfn} \label{genericaction.def} 
Let $X$ be a complete $T$-variety. 
Let $\mu:T\times X\to X$ be the action map.
We say that $\mu$ is a {\boldmath $T$\bf-skeletal action} if 
%
\begin{enumerate}
 \item $X^T$ is finite, and
 \item The number of one-dimensional orbits of $T$ on $X$ is finite.
\end{enumerate}
In this context, $X$ is called a {\boldmath$T$\bf-skeletal variety}. 
The associated graph of fixed points and invariant curves is called 
the {\bf GKM graph} of $X$. We shall 
denote this graph by $\Gamma(X)$.  
\end{dfn}

\begin{ex}
Smooth $T$-skeletal varieties 
include regular compactifications of reductive groups 
(\cite{bif:reg}, \cite{lp:equiv}) 
and, more generally, 
regular compactifications of symmetric varieties of minimal rank. 
The Chow rings of these varieties are 
described in \cite{bj:chern} by means of GKM theory.   
In constrast,  
Schubert varieties and projective group embeddings
of reductive groups are examples of {\em singular} 
$T$-skeletal varieties.  
The former have a paving by affine spaces
and their equivariant cohomology is well-known \cite{c:schu}. 
The latter are spherical varieties and, when rationally smooth, 
their equivariant cohomology has been described 
by the author in \cite{go:equiv}.  
Our version of GKM theory (Theorem \ref{gkm.thm}) 
will generalize these topological  
descriptions to the corresponding 
equivariant operational $K$-theory rings 
(Example 5.5 and Section 6).   

\end{ex}

\smallskip

Let $X$ be a complete $T$-variety and let 
$C$ be a $T$-invariant irreducible curve of $X$, 
which is not fixed pointwise by $T$. 
Let $\pi:\tilde{C}\to C$ be the ($T$-equivariant) normalization.  
Then $\tilde{C}$ is isomorphic to $\P^1$. 
Denote by $0,\infty$ the two fixed points 
of $T$ in $\tilde{C}$, and denote by $x_0,x_\infty$ 
their corresponding images 
via $\pi$. Then 
$\tilde{C}\setminus \{0,\infty\}=C\setminus \{x_0,x_\infty\}$ 
identifies to $k^*$, where $T$ acts on 
$\tilde{C}\setminus \{0,\infty\}$ 
via a unique character $\chi$ 
(when interchanging $0$ and $\infty$, 
one replaces $\chi$ by $-\chi$). 
Clearly, 
$T$ has either one or two fixed points in $C$.  


Notice that, in principle, Definition \ref{genericaction.def} 
allows for $T$-invariant irreducible curves with exactly one 
fixed point (i.e. the GKM graph 
$\Gamma(X)$ may have simple loops). 
We shall see that the functor $\K(-)$ 
``contracts'' such loops 
to a point. 

\begin{prop}\label{1.2.fix.point.curves.prop}
Let $X$ be a complete $T$-variety and let $C$ 
be a $T$-invariant irreducible curve of $X$ which is not 
fixed pointwise by $T$. Then the image of the injective map 
$i^*_T:\K(C)\to \K(C^T)$ is described as follows: 
\begin{enumerate}[(i)]
 \item If $C$ has only one 
fixed point, say $x$, then $i^*_T:\K(C)\to \K(x)$ is 
an isomorphism; that is, $\K(C)\simeq R(T)$. 
\smallskip

\item If $C$ has two fixed points, then  
$$\K(C)\simeq 
\{(f_0,f_\infty) \in R(T)\oplus R(T)\,|\, 
f_0\cong f_\infty \mod 1-e^{-\chi}\},$$
where $T$ acts on $C$ via the character $\chi$. 
\end{enumerate}
\end{prop}

\begin{proof} 
Let $\pi:\P^1\to C$ be the normalization map. 
By \cite[Theorem 2]{vv:hkth} (see also 
\cite[Theorem 1.3]{uma:kth})  
$$K_T(\P^1)=\{(f_0,f_\infty)\in R(T)\oplus R(T)\,|\, 
f_0\cong f_\infty \mod 1-e^{-\chi}\},$$ 
where $\chi$ is the character of the $T$-action on $C$. 
Moreover, given that $\P^1$ is smooth, we get $\K(\P^1)=K_T(\P^1)$. 
In view of this, and Gillet-Kimura criterion 
(Theorem \ref{gillet.thm}), it suffices
to find 
the image of the injective map 
$\pi^*:\K(C)\to \K(\P^1)$ explicitly.  
First, assume that $C$ has only one fixed point, 
say $x=\pi(0)=\pi(\infty)$. Then an element $f\in \K(\P^1)$ 
is in the image of $\pi^*$ if and only if 
the restriction $(f_0,f_\infty)\in \K(\{0,\infty\})$ 
is in the image of the induced map 
$\pi^*:\K(x)\to \K(\{0,\infty\})$. But the latter 
morphism  
is simply the diagonal inclusion, so we get that $f\in \K(\P^1)$ is in the image of 
$\pi^*$ if and only if $f_0=f_{\infty}$.  
Therefore, $\K(C)=R(T)$ and $i^*_T:\K(C)\to \K(x)$ 
is an isomorphism. Finally, if $\pi(0)\neq \pi(\infty)$, 
a similar analysis yields 
assertion (ii).  
\end{proof}

Let $X$ be a $T$-skeletal variety. 
Now, as done in the Introduction, it is 
possible to define a ring $PE_T^*(X)$ 
of {\bf piecewise exponential functions}. 
We recall the construction here, taking into account  
Proposition \ref{1.2.fix.point.curves.prop}.
Let $K_T(X^T)=\oplus_{x\in X^T}R_x$, 
where $R_x$ is a copy of the representation ring $R(T)$. 
We then define $PE_T(X)$ as the subalgebra of $K^0_T(X^T)$ 
given by 
$$
PE_T(X)=\{(f_1,\ldots, f_m)\in \oplus_{x\in X^T}R_x\,|\, f_i\equiv f_j \mod 1-e^{-\chi_{i,j}}\}
$$ 
where $x_i$ and $x_j$ are the two (perhaps equal)
fixed points 
in the closure of 
the one-dimensional 
$T$-orbit $C_{i,j}$, 
and $\chi_{i,j}$ 
is the character of $T$ associated with $C_{i,j}$. 
This character is uniquely determined up to sign 
(permuting the two fixed points changes $\chi_{i,j}$ to its opposite). 
Invariant curves with only one fixed point 
do not impose any relation (this  
is compatible with Proposition \ref{1.2.fix.point.curves.prop}).


\begin{thm}\label{gkm.thm}
Let $X$ be a complete $T$-skeletal variety. 
Then the pullback 
$i^*_T:\K(X)\to \K(X^T)$ 
induces an isomorphism between 
$\K(X)$ and $PE_T(X)$. 
\end{thm}




\begin{proof}
Observe that a codimension one subtorus of $T$ is the kernel 
of a primitive (i.e. indivisible) character of $T$. Such character 
is uniquely defined up to sign.

Let $\pi$ be a primitive character of $T$. 
Let $X^{\ker {\pi}}=\bigcup_j X_j$ be the 
decomposition into irreducible components. 
Notice that each $X_j$ is either a fixed point, 
or a    
$T$-invariant irreducible curve.    
Now, with the notation of Corollary \ref{irr.seq.cor}, we have 
the commutative diagram 
$$
\xymatrix{
0\ar[r] & \K(X^{\ker {\pi}})\ar[r] \ar[d]_{i^*_{T,\ker {\pi}}}&
\bigoplus_i \K(X_j) \ar[r]\ar[d]& 
\bigoplus_{i,j}\K(X_{i,j})\ar[d]^{\rm id}\\
0\ar[r] & \K(X^T)\ar[r] &
\bigoplus_i \K(X_j^T) \ar[r]& \bigoplus_{i,j}\K(X_{i,j}),
}
$$
where each $X_{i,j}$ is just a fixed point, and $T$ acts 
on those $X_j$'s that are curves via a character $\chi_j$, 
a multiple of $\pi$. The image of the 
middle vertical map is completely characterized by 
Proposition \ref{1.2.fix.point.curves.prop}, and so 
is the image of $i^*_{T,\ker {\pi}}$, as it follows from  
Remark \ref{normalization.irr.comp.rem}. In short,  
${\rm Im}(i^*_{T,\ker {\pi}})\simeq PE_T(X^{\ker {\pi}})$. 
Now apply Theorem \ref{cs.thm} to conclude the proof. 
\end{proof}



Let $X$ be a $T$-skeletal variety. 
Notice that $\Gamma(X)$  
is a singular projective $T$-variety with the 
same equivariant  
operational $K$-theory as that of $X$. In symbols, 
$\K(\Gamma(X))=\K(X)$. This is 
simply  a rephrasing of 
Theorems \ref{cs.thm} and \ref{gkm.thm}.  

\begin{ex} (Bruhat graph.)
Let $G$ be a connected reductive group with Borel subgroup $B$ 
and maximal torus $T\subset B$. Let $W$ be the Weyl group of 
$(G,T)$. It is a finite group generated by reflections 
$\{s_\alpha\}_{\alpha \in \Phi}$, where $\Phi$ 
stands for the set of roots of $(G,T)$. 
Now let $X(w)=\overline{BwB/B}\subset G/B$  
be the Schubert variety associated to $w\in W$. 
In what follows we extend 
the usual picture of $K_T(G/B)$
to $\K(X(w))$. 
Denote by $I_w$ the Bruhat interval 
$$[1,w]=\{x\in W\,|\,x\leq w\}.$$ 
Notice that $X(w)^T=I_w$. 
As a shorthand, set $R=R(T)$.  
Then, 
%
by Theorem \ref{gkm.thm}, 
$\K(X(w))$ is the subring of 
$\oplus_{x\in I_w}R$
consisting  of all $\sum_{x\in I_w}f_x x$ 
such that $f_x\cong f_{s_\alpha x}\mod 1-e^{-\alpha}$,  
whenever $(i)$ $s_\alpha$ is a reflection of $W$ and $(ii)$
$x,s_\alpha x \in I_w$. Finally,   
$\K(X(w))$ is a free $R(T)$-module of rank $|I_w|$. 
This is a consequence of equivariant Kronecker 
duality (Property (f), Introduction) together 
with the fact that  
$K^T(X(w))$ 
is a free $R(T)$-module of rank $|I_w|$ (for   
$X(w)$ has a paving by affine spaces, 
cf. \cite[Lemma 1.6]{uma:kth}). 
\end{ex}

\begin{rem}
Let $X$ be a $T$-skeletal variety. By Theorem 
\ref{gkm.thm}, the $R(T)$-algebra $\K(X)$ identifies to 
$PE_T(X)\subset R(T)^m$. Moreover, $PE_T(X)$ and $R(T)^m$ 
have the same quotient field (by the localization theorem).
It follows that $PE_T(X)$ is a reduced, finitely generated 
$\Z$-algebra. The same holds for the natural extension 
$PE_T(X)_k:=PE_T(X)\otimes k$, a $k$-algebra of  
dimension $d=dim(T)$. Let $V(X)$ be the corresponding 
affine $k$-variety (defined over a finite algebraic extension 
of $\Q$). It is worth noting that 
the associated map $i_T^*:\K(X)_k\to k[T]^m$ is the normalization 
(cf. \cite[Proposition 2]{bri:pd}). For this, first 
observe that 
$k[T]$ is a subring of $\K(X)_k$, as a choice of fixed 
point yields a section of the structural map $\K(X)_k\to \K(pt)_k\simeq k[T]$. 
Secondly, $k[T]^m$ is a finite module over $k[T]$, so it is also a finite 
module over 
$\K(X)$.   
Finally, since $k[T]^m$ is integrally closed 
in its quotient field and   
 $\K(X)_k$ and $k[T]^m$ have the 
same quotient field (by the localization theorem), we conclude that  
$i^*$ is the normalization. 
Hence, the normalization of the affine variety $V(X)$ is the union of $|X^T|$
disjoint copies of $T$. 
Moreover, the set $V(X)$ is obtained as follows: for any  
character $\chi$ associated to a $T$-invariant curve 
with fixed points $x$ and $y$ 
we identify the toric hyperplanes 
$\{1-e^{\chi}=0\}$ in $T_x$ and $T_y$, provided $x\neq y$. 
If the aforementioned $x$ and $y$ are the same, 
then we set $T_x=T_y$, in accordance with Proposition \ref{1.2.fix.point.curves.prop}.   
\end{rem}

\section{Equivariant operational $K$-theory rings of projective group embeddings}

Throughout this section  
we denote by  
$G$ 
a connected reductive 
linear algebraic group (over $k$) with Borel subgroup 
$B$ and maximal torus $T\subset B$.  
We denote by $W$ the Weyl group of $(G,T)$. 
Observe that $W$ is generated by reflections $\{s_\alpha\}_{\alpha \in \Phi}$, 
where $\Phi$ stands for the set of roots of $(G,T)$. We write  
$U_\alpha$ for the unipotent subgroup of $G$ associated to $\alpha \in \Phi$. 
Since $W$ acts on 
$\Delta$, the character group of $T$, 
there is a natural action of $W$ on $\Z[\Delta]$ 
given by $w(e^\lambda)=e^{w(\lambda)}$, for each 
$w\in W$ and $\lambda \in \Delta$. 
Recall that we can identify 
$R(G)$ with $R(T)^W$ via restriction to $T$, 
where 
$R(T)^W$ denotes the subring of $R(T)$ invariant under the action of $W$. 


\smallskip

An affine algebraic monoid $M$ is called {\em reductive} it is 
irreducible, normal, and its unit group is a reductive 
algebraic group. See \cite{re:lam} for many details. 
Let $M$ be a reductive monoid with zero and unit group $G$.
We denote by $E(\overline{T})$ the idempotent set of the associated 
affine torus embedding $\overline{T}$,  
that is, $E(\overline{T})=\{e\in \overline{T}\,|\, e^2=e\}.$
One defines a partial order on $E(\overline{T})$ by declaring 
$f\leq e$ if and only if $fe=f$.
Denote by $\Lambda \subset E(\overline{T})$, the cross section lattice of $M$.
The Renner monoid $\mathcal{R}\subset M$ is a finite monoid
whose group of units is $W$ 
and contains $E(\overline{T})$
as idempotent set. In fact, 
any $x\in \mathcal{R}$ can be written as $x=fu$, 
where $f\in E(\overline{T})$ and $u\in W$. 
Given $e\in E(\overline{T})$, we write $C_W(e)$ for 
the centralizer of $e$ in $W$. 
Denote by $\mathcal{R}_k$ the set of elements of rank $k$ in $\mathcal{R}$, 
that is, 
$\mathcal{R}_k=\{x\in \mathcal{R} \, | \, \dim{Tx}=k \,\}.$
Analogously, one has 
$\Lambda_k\subset \Lambda$ and $E_k\subset E(\overline{T})$. 
%

\smallskip


A normal irreducible variety $X$ 
is called an {\em embedding} of $G$, 
or a {\em group embedding}, 
if $X$ is a $G\times G$-variety containing an open orbit 
isomorphic to $G$. Due to 
the Bruhat decomposition, group embeddings are spherical 
$G\times G$-varieties. 
Substantial information about the topology of a group 
embedding can be obtained by restricting one's attention 
to the induced action of $T\times T$.  
When $G=B=T$, we get back the notion of toric varieties.  
%
%
Let $M$ be a reductive monoid with zero and unit group $G$.
Then there exists a central one-parameter subgroup 
$\epsilon:\mathbb{G}_m^*\to T$,
with image $Z$, 
such that $\displaystyle \lim_{t\to 0}\epsilon(t)=0$. 
Moreover, 
the quotient space $$\P_\epsilon(M):=(M\setminus\{0\})/Z$$
is a normal projective variety on which $G\times G$ acts via
$(g,h)\cdot [x]=[gxh^{-1}].
$
Hence, $\P_\epsilon(M)$ is a normal projective embedding of the quotient group $G/Z$.
These varieties were introduced 
by Renner in his study of algebraic monoids 
(\cite{re:hpoly}, \cite{re:ratsm}). 
Notably, normal projective embeddings of  
connected reductive groups are exactly the 
projectivizations of normal algebraic monoids 
\cite{ti:sph}.


\medskip

Now let $X=\P_\epsilon(M)$ be a (projective) group embedding. In 
\cite{go:equiv} we compute the finite GKM data
coming from the $T\times T$-fixed points and $T\times T$-invariant curves of 
$X$ in terms of the combinatorial 
invariants of $M$. 
These computations are independent of whether or 
not $X$ is rationally smooth. 

\begin{thm}[\protect{\cite[Theorems 3.1, 3.5]{go:equiv}}]\label{gkm_data_emb.thm} 
Let $X=\mathbb{P}_\epsilon(M)$ be a projective group embedding. 
Then its natural $T\times T$-action
\[
\mu:T\times T\times \mathbb{P}_\epsilon(M)\to \mathbb{P}_\epsilon(M),\; \; \;
(s,t,[x])\mapsto [sxt^{-1}]
\]
is $T\times T$-skeletal. Indeed, after identifying the
elements $x$ of $\mathcal{R}_1$ with their corresponding images $[x]$ in $X$, 
the set $X^{T\times T}$ corresponds to $\mathcal{R}_1$. 
As for the closed $T\times T$-curves of $X$, 
they fall into three types:
\begin{enumerate}
\item $\overline{U_\alpha [ew]}$, $e\in E_1(\overline{T})$, $s_\alpha\notin C_W(e)$ and $w\in W$.
\item $\overline{[we]U_\alpha}$, $e\in E_1(\overline{T})$, $s_\alpha\notin C_W(e)$ and $w\in W$.
\item $\overline{[TxT]}=\overline{[Tx]}=\overline{[xT]}$, where $x\in\mathscr{R}_2$. 
\end{enumerate}
\end{thm}

The curves of type 1 and 2 lie entirely in closed $G\times G$-orbits, 
whereas the curves of type 3 do not. Curves of type 3 can be further
separated into whether or not the corresponding 
$T\times T$-fixed points are in the same closed $G\times G$-orbit. 
In \cite[Section 4]{go:equiv}, we identify explicitly
the $T\times T$-characters associated to these curves. 
With such data at our disposal, 
Theorem \ref{gkm.thm} 
yields an immediate 
translation of 
\cite[Theorem 4.10]{go:equiv} 
into the language of 
equivariant operational $K$-theory. Furthermore,  
as Theorem \ref{gkm.thm} does not 
require any conditions on the singular locus, 
the result (Theorem \ref{gkm.emb.thm}) 
applies to all projective group embeddings.  
This description  
coincides with that of Uma (\cite[Theorem 2.1]{uma:kth})
when $X=\P_\epsilon(M)$ is smooth, 
and it extends our previous work on 
rationally smooth group embeddings \cite{go:equiv}. 
To state it, we record a few extra facts. 
Let $\Lambda_1$ be the set of rank-one idempotents of the cross-section lattice $\Lambda$. 
Each closed $G\times G$-orbit of $X=\P_\epsilon(M)$ can be written uniquely as $G[e]G\simeq G/P_e\times G/P^-_e$,
where $e\in \Lambda_1$, and $P_e$, $P^-_e$ are opposite 
parabolic subgroups (see e.g. \cite{re:lam}).  


\begin{thm}\label{gkm.emb.thm} 
Let $X=\mathbb{P}_\epsilon(M)$ be a group embedding. Then the natural map 
      $$\KTT(X) \longrightarrow \KTT\left( \bigsqcup_{e\in \Lambda_1} G[e]G \right)
=\bigoplus_{e\in \Lambda_1}K_{T\times T}(G[e]G)$$ 
      is injective. In fact,  
       its image consists of all tuples $(\varphi_e)_{e\in \Lambda_1}$, indexed over $\Lambda_1$ and with 
       $\varphi_e \in K_{T\times T}(G[e]G)$, 
       subject to the additional conditions:
   \medskip
   \begin{enumerate}[(a)]
          \item If $f\in E_2(\overline{T})$ and 
there is a (necessarily unique) reflection $s_{\alpha_f}$  such that 
$s_{\alpha_f} f=fs_{\alpha_f}\neq f$, then
                $$\varphi_{e_f}(f_1 u)\equiv \varphi_{e_f}(f_2 u) \mod 1-e^{-\alpha_f}\otimes e^{-(\alpha_f \circ {\rm int}(u))},$$
                for all $u\in W$. 
                Here, $f_1$ and $f_2=s_{\alpha_f}\cdot f_1 \cdot s_{\alpha_f}$ are the two idempotents in $E_1(\overline{T})$ below $f$, 
                the root $\alpha_f$ corresponds to the reflection $s_{\alpha_f}$,
                and 
                $e_f \in \Lambda_1$ is the unique element of $\Lambda_1$ which is conjugate to $f_1$. 
                
\smallskip
           \item If $f\in E_2(\overline{T})$ and $sf=fs=f$ for every reflection 
$s\in W$, then 
                 $$\varphi_{e_1}(f_1 u)\equiv \varphi_{e_2}(f_2 u) \mod 1-e^{-\lambda_f}\otimes e^{-(\lambda_f \circ int(u))},$$
                 for all $u\in W$. Here,  
                 $\lambda_f$ is the character of $T$ defined by the composition 
                 $$T\to Tf\to Tf/k^*\simeq k^*, $$  
                 the idempotents $f_1, f_2$ are the unique idempotents below $f$,
                 and $e_i\in \Lambda_1$ is conjugate to $f_i$, for $i=1,2$. 
    \end{enumerate} 
\end{thm}

\begin{proof}
Since $X^{T\times T}\subset \bigsqcup_{e\in \Lambda_1} G[e]G$, 
Corollary \ref{injectivity.cor} renders the natural map 
$\KTT(X) \to \KTT\left( \bigsqcup_{e\in \Lambda_1} G[e]G \right)$ 
injective. Moreover, $G[e]G$ is smooth, so 
${\rm op}K_{T\times T}(G[e]G)$ is 
isomorphic to $K_{T\times T}(G[e]G)$. 
Finally, we apply Theorem \ref{gkm.thm},  
taking into account that:   

\noindent (i) the curves of type 1 and 2 in 
Theorem \ref{gkm_data_emb.thm} are contained in $\bigsqcup_{e\in \Lambda_1} G[e]G$ 
and these curves describe $K_{T\times T}(G[e]G)$ via Example 5.5, 

\noindent (ii) the characters associated to the curves of type (3) give 
assertions (a) and (b), as in 
\cite[Theorem 4.10]{go:equiv}.  
\end{proof}

\medskip

If $X=\P_\epsilon(M)$ is a group embedding, then 
$X$ is $G\times G$-spherical. If moreover $\pi_1(G)$ 
is torsion free, then  
Corollary \ref{ring.W.inv.cor} states that  
${\rm op}K_{G\times G}(X)$ can be 
read off from $\KTT(X)$ by computing invariants:
$$
K_{G\times G}(X)\simeq \KTT(X)^{W\times W}.
$$


\begin{cor}\label{GG_opk_emb.cor}
Let $X=\mathbb{P}_\epsilon(M)$ be a group embedding. 
If $\pi_1(G)$ is torsion free, then  
the ring ${\rm op}K_{G\times G}^*(X)$ consists of 
all tuples $(\Psi_e)_{e\in \Lambda_1}$, where 
$$\Psi_e:WeW \to (R(T) \otimes R(T))^{C_W(e)\times C_W(e)},$$
such that 
\begin{enumerate}[(a)]
 \item If $f \in E_2(\overline{T})$ and $H_f=\{f,s_{\alpha_f} f\}$, then  
       $$\Psi_e(f_1 )\equiv \Psi_e(f_2 ) \mod 1-e^{-\alpha_f}\otimes e^{-\alpha_f},$$ 
       where $e\in \Lambda_1$ is conjugate to $f_1$, 
       $f_2=s_{\alpha_f}\cdot f_1\cdot s_{\alpha_f}$, the reflection $s_{\alpha_f} \in C_W(f)$ is associated with the root $\alpha_f$, 
       and $f_i\leq f$.

 \item If $f\in E_2$ and $H_f=\{f\}$, then 
       $$\Psi_e(f_1)\equiv \Psi_{e'}(f_2) \mod 1-e^{-\lambda_f}\otimes e^{-\lambda_f},$$
       where $\lambda_f$ is character of $T$ defined by $f$, 
       and $f_1, f_2\leq f$ are conjugate to $e$ and $e'$, respectively. 
\end{enumerate}
\end{cor}

\begin{proof}
Simply adapt the proof of \cite[Corollary 4.11]{go:equiv}, 
using Theorem \ref{gkm.emb.thm} 
and Corollary A.4.  
\end{proof}

\smallskip

Associated to $X=\P_\epsilon(M)$, there is a
projective torus embedding $\mathcal{Y}$ of $T/Z$, namely, 
$$\mathcal{Y}=\P_\epsilon(\overline{T})=[\overline{T}\setminus \{0\}]/Z.$$
By construction, $\mathcal{Y}$ is a normal projective toric variety and $\mathcal{Y}\subseteq X$. 
Our next theorem 
allows to compare the equivariant operational $K$-theories of 
$X$ 
and its associated torus embedding $\mathcal{Y}\subseteq X$.  
The situation for general group embeddings contrasts deeply with 
the corresponding one for regular embeddings 
(\cite[Corollary 3.1.2]{bri:bru}, \cite[Corollary 2.2.3]{uma:kth}). 


\begin{thm}\label{comparison.emb.thm}
Notation being as above, if $\pi_1(G)$ is torsion free, then 
the inclusion of the associated torus embedding $\iota:\mathcal{Y}\hookrightarrow X$ 
induces an injection:
$$
\xymatrix{\iota^*:{\rm op}K_{G\times G}^*(X) \ar@{^(->}[r] & \KTT(\mathcal{Y})^W\simeq (\K(\mathcal{Y})\otimes R(T))^W,}
$$
where the $W$-action on $\KTT(\mathcal{Y})$ is induced from the action of {\rm diag}$(W)$ on $\mathcal{Y}$. 
Moreover, $\iota^*$ is an isomorphism whenever 
$C_W(e)=\{1\}$ for every $e\in \Lambda_1$. 
\end{thm}

\begin{proof}
The argument here is an adaptation of 
\cite[Proof of Theorem 4.12]{go:equiv}. 
First, consider the commutative diagram
$$
\xymatrix{
\KTT(X)  \ar@{^(->}[r]           \ar[d]^{\iota^*} &  \KTT(X^{T\times T}) \ar[d]^{\iota^*} \\
\KTT(\mathcal{Y}) \ar@{^(->}[r]              & \KTT({\mathcal{Y}}^{T\times T})_,                 \\
}
$$ 
where both horizontal maps are injective (Proposition \ref{inj.fix.set.lem}). 
On the other hand, 
recall that $\Lambda_1$ provides a set of representatives of 
both the $W\times W$-orbits in $X^{T\times T}=\mathcal{R}_1$
and the $W$-orbits 
in $\mathcal{Y}^{T\times T}=E_1(T)$. 
Thus, after taking 
invariants, we obtain an injection
$$\xymatrix{
\KTT(\mathcal{R}_1)^{W\times W}=\bigoplus_{e\in \Lambda_1}((R(T)\otimes R(T))^{C_W(e)\times C_W(e)} \ar@{^(->}^{\iota^*}[d]\\ 
\KTT(E_1(T))^W=\bigoplus_{e\in \Lambda_1}(R(T)\otimes R(T))^{C_W(e)}_.}
$$

Placing this information into the commutative square above   
renders the map  
$$\iota^*: \KTT(X)^{W\times W} \longrightarrow \KTT(\mathcal{Y})^W$$ 
injective. 
Now observe that $\KTT(\mathcal{Y})^W\simeq (\K(\mathcal{Y})\otimes R(T))^W$.
Truly, we have a split exact sequence
$$
\xymatrix{
1 \ar[r]& diag(T) \ar[r]& T\times T \ar[rrr]^{(t_1,t_2)\mapsto t_1t_2^{-1}}& & & T \ar@/^1pc/ @{-->}[lll] \ar[r]& 1,
}
$$
where the splitting is given by $t\mapsto (t, 1)$.
It follows that $T\times T$ is canonically isomorphic to $diag(T)\times(T\times {1})$. 
Clearly, 
$diag(T)$ acts trivially on $\mathcal{Y}$. 
Hence,  
by \cite[Corollary 5.5]{ap:opk}, we have a
ring isomorphism $\KTT(\mathcal{Y})\simeq {\rm op}K_{diag(T)}\otimes \K(\mathcal{Y})$. 
This isomorphism is in fact  
$W$-invariant (because the $W$-action on the 
operational rings is induced from
the action of $diag(W)$ on $\mathcal{Y}$).


For the second assertion, assume that 
$C_W(e)=\{1\}$ for all $e\in \Lambda_1$. 
We need to show that $\iota^*$ is surjective. To achieve our goal, we modify
slightly an argument of \cite{lp:equiv}, Section 4.1, and Brion \cite{bri:bru}, Corollary 3.1.2.
Define the $T\times T$-variety 
$$\mathcal{N}=\bigcup_{w\in W} w\mathcal{Y}.$$
We claim that this union is, in fact, a disjoint union.
Indeed, observe that $\mathcal{N}$ contains all the $T\times T$-fixed points of $X$.
That is, $\mathcal{N}$ has $|\mathcal{R}_1|$ fixed points.
On the other hand, each $w\mathcal{Y}$ has $|E_1|$ fixed points (for its corresponding $T$-action). 
Now, if it were the case that 
there is a pair of distinct subvarieties $w\mathcal{Y}$ and $w'\mathcal{Y}$
with non-empty intersection, then 
this intersection should also contain $T\times T$-fixed points.
But then a simple counting argument would yield $|\mathcal{R}_1|<|E_1||W|$. 
This is impossible, by our assumptions and \cite[Lemma 4.14]{go:equiv}.
Hence,  
$$\mathcal{N}=\bigsqcup_{w\in W} w\mathcal{Y}.$$

In this setup, Corollary \ref{injectivity.cor} implies that 
the restriction map 
$$\KTT(X)\to \KTT(\mathcal{N}).$$
is injective. 
From Theorem \ref{gkm_data_emb.thm} we know 
that all the $T\times T$-curves of $X$ are contained either
in closed $G\times G$-orbits (curves of type $1.$ and $2.$) or
in $\mathcal{N}$ (curves of type $3.$). 
Moreover, note that the curves 
of type $3.$ are exactly the $T\times T$-invariant curves of $\mathcal{N}$,
so $\mathcal{N}$ is $T\times T$-skeletal and Theorem \ref{gkm.thm}
applies  to it.
After taking $W\times W$-invariants (cf. Corollary \ref{GG_opk_emb.cor}),  
we see that the aforementioned 
map 
induces an isomorphism
$$\KTT(X)^{W\times W}\simeq \KTT(\mathcal{N})^{W\times W}\simeq 
\left(\bigoplus_{w\in W} \KTT(\mathcal{Y})\right)^{W\times W}
\simeq \KTT(\mathcal{Y})^W_.$$ This concludes the proof.  
\end{proof}




\smallskip

\begin{lem}[\protect{\cite[Lemma 4.14 and Corollary 4.15]{go:equiv}.}] \label{toroidal.lem}
Let $X=\P_\epsilon(M)$ be a group embedding. 
Then the following are equivalent: 
\begin{enumerate}[(a)]
 \item $C_W(e)=\{1\}$ for every $e\in E_1(\overline{T})$.
 \item All closed $G\times G$-orbits in $X$ are isomorphic to $G/B\times G/{B^-}$. \hfill $\square$
\end{enumerate}
\end{lem} 

Group embeddings satisfying the equivalent conditions of 
Lemma \ref{toroidal.lem} are called {\bf toroidal} embeddings (see e.g. 
\cite[Chapter 5]{ti:sph}).  
Furthermore, 
smooth toroidal embeddings are exactly the regular embeddings of reductive groups  
\cite[Theorem 29.2]{ti:sph}. 

\smallskip

Theorem \ref{comparison.emb.thm} gives an explicit relation 
between our results and those of \cite{ap:opk}. Indeed, 
if $X=\P_\epsilon(M)$ is a toroidal group embedding and $\pi_1(G)$ 
is torsion free, then 
${\rm op}K_{G\times G}(X)$ is isomorphic to the subring of $W$-invariants in  
$\K(\mathcal{Y})\otimes R(T)$, where 
$\K(\mathcal{Y})$ is the ring of 
integral piecewise exponential functions on the fan of 
$\mathcal{Y}$.

%



\section{Further remarks}

\noindent (1) {\em Extending the results to equivariant 
operational Chow groups. Poincar\'e duality for singular schemes.} 
Kimura's 
cohomological descent for  
envelopes \cite[Theorems 2.3 and 3.1]{ki:op} 
has also been established for equivariant operational 
Chow groups ${\rm op}A^*_T(-)$ \cite[Section 2.6]{eg:eqint},       
the operational cohomology groups  
associated to Edidin and Graham's  
equivariant Chow groups $A^T_*(-)$.  
On the category of smooth schemes, the functors
 ${\rm op}A^*_T(-)$ 
and $A^T_*(-)$ 
are known to agree ({\em op. cit.}). 
Furthermore, on the subcategory of smooth projective $T$-schemes,  
corresponding  versions of the localization theorem 
and CS property   
hold for $A^T_*(-)_\Q$  
\cite[Section 3]{bri:eqchow}.  
Since these are the intersection theory 
analogues of our main tools,    
%
our arguments are readily translated into 
the language of equivariant operational Chow groups 
with $\Q$-coefficients, 
yielding versions of 
Theorems \ref{localization.thm},  
\ref{cs.thm} and 
\ref{gkm.thm} applicable to 
all singular complete $T$-schemes. 
See \cite{go:tlinear} for a slightly different 
approach in the case of $T$-linear varieties, and 
%
%
%
\cite{go:rm} for some applications to 
characterizing 
Poincar\'e duality on the Chow groups of     
singular $T$-schemes.  
%
Moreover, when $k=\C$ and $X$ is a
rationally smooth $T$-skeletal variety, 
there is a natural 
isomorphism between ${\rm op}A^*_T(X)_\Q$ and $H^*_T(X)_\Q$,  
as their images on $A^*_T(X^T)_\Q=H^*_T(X^T)_\Q$ 
are canonically isomorphic \cite{gkm:eqc}, \cite{go:cells}. 
In particular, the 
equivariant operational Chow groups of (complex) 
rationally 
smooth $T$-skeletal varieties 
are free modules over ${\rm Sym}[\Delta]_\Q$.  
From this point of view, 
equivariant operational Chow groups behave like 
equivariant intersection cohomology, 
though the former
are somewhat more combinatorial and 
easier to compute on $T$-skeletal varieties.  
Notably, for projective group embeddings, 
the results of \cite{go:equiv} remain valid 
when translating them into 
the context of rational smoothness in Chow groups \cite{go:rm} 
and equivariant operational 
Chow rings. 
The results will appear in \cite{go:opemb}.     


\medskip

\noindent (2) {\em Equivariant multiplicities in $K$-theory.} 
Let $X$ be a complete $T$-scheme with finitely 
many fixed points.  
In virtue of Thomason's localization theorem for $K^T(-)$ 
\cite[Theorem 2.1]{th:con}, the following 
identity holds in $\mathcal{Q}(\Delta)$, the quotient field of $\Z[\Delta]$: 
$$
[\O_X]=\sum_{x\in X^T}EK(x,X)[\O_x], 
$$
where the various $EK(x,X)$ are (possibly zero) 
rational functions on $(\Delta_\Q)^*$.
Following the nomenclature of 
\cite[Section 4.2]{bri:eqchow} we call 
$EK(x,X)$ the {\em $K$-theoretic 
equivariant multiplicity of $X$ at $x$}.
If all the $EK(x,X)$ are non-zero, 
then, by Theorem \ref{localization.thm}, 
the Poincar\'e duality map 
$$
\K(X)\to K^T(X), \;\;\; c\mapsto c_{id_X}(\O_X)
$$ 
is injective (cf. proof of \cite[Theorem 4.1]{bri:pd}).  
We anticipate  
that the $EK(x,X)$'s are non-zero whenever 
$x$ is an attractive fixed point of $X$, because, in that case,     
$EK(x,X)$ 
is related to the Hilbert series of ${\rm Proj}(k[X_x])$, 
where $X_x$ is the unique open affine $T$-stable
neighborhood of $x$ (cf. \cite{bri:eqchow}, \cite{re:hilb}, \cite{bv:rr}). 
The notion of $K$-theoretic equivariant 
multiplicity at attractive fixed points 
is already present in 
the study 
of flag varieties 
(see e.g. \cite{bbm:nil}).    
For complete toric varieties and simple group embeddings, 
our claim would imply that the natural map 
$\K(X)\to K^T(X)$ is always injective (this deeply 
contrasts with the behaviour of the map 
$K_T(X)\to K^T(X)$, 
whose kernel could be rather large, cf. \cite{ap:opk}). 
In contrast, 
surjectivity of the Poincar\'e duality map 
on singular schemes is a more delicate property, 
and quite often it does not hold. 
For instance, 
consider the 
$\G$-action on $\P^3$ given by 
$t\cdot [x,y,z,w]=[t^2x,t^4y,t^3z,w]$.
Now let $Y\subset \P^3$ be the projective surface 
$z^2=xy$. 
Clearly $Y$ is $\G$-invariant,   
$\K(Y)$ is torsion free, but $K^T(Y)$ 
has $R(\G)$-torsion 
coming from the fact that 
$\mu_2\subset \G$ fixes two lines in $Y$. 
%
%
We shall develop these ideas 
and explore the behaviour of 
$K$-theoretic equivariant multiplicities  
in a subsequent paper.     



%
%

%
%
%
%

\appendix
\section{
$G$-equivariant K\"{u}nneth formula for spherical varieties}\label{app:A}

Recall that a $G$-variety is 
called {\em spherical} if it 
contains a dense $B$-orbit. 
Examples include flag varieties, 
symmetric spaces, and $G\times G$-equivariant  
embeddings of $G$ (e.g. toric varieties are spherical). 
For an up-to-date discussion of 
spherical varieties, 
see \cite{ti:sph} and the 
references therein.

\smallskip

The following is a result of Merkurjev \cite{mer:comp}. 

\begin{thm}\label{merkurjev.thm} 
Let $G$ be a connected reductive group. Suppose that 
$\pi_1(G)$ is torsion-free. Then the following hold: 
\begin{enumerate}[(i)]
 \item $R(T)$ is a free $R(G)$-module of rank $|W|$, and 
       $R(T)\simeq R(G)\otimes \Z[W]$. 
 \item If $X$ is a $G$-scheme, then 
$$R(B)\otimes_{R(G)}K^G(X)\simeq K^G(X\times G/B)\simeq K^B(X)\simeq K^T(X).$$ 
In particular, $K^G(X)\simeq K^T(X)^W$. \hfill $\square$
\end{enumerate}
\end{thm}


\begin{thm} \label{GequivKunneth.thm}
Let $G$ be a connected reductive group with $\pi_1(G)$ torsion free. 
Let $X$ be a $G$-spherical variety. Then for any $G$-variety $Y$  the 
exterior product map, or K\"{u}nneth map, 
$$Ku_G:K_G(X)\otimes_{R(G)}K_G(Y)\to K_G(X\times Y)$$ 
is an isomorphism.  
\end{thm}

\begin{proof}
Consider the commutative diagram  
$$
\xymatrix{
K_B(X)\otimes_{R_B}K_B(Y) \ar[r]^{{\rm Ku}_B}& K_B(X\times Y) \\
K_G(X\times G/B)\otimes_{R_B}K_G(Y\times G/B) \ar[u]^{p_1\otimes p_2}& K_G((X\times Y)\times G/B)\ar[u]^{p}\\
[R_B\otimes_{R_G}K_G(X)]\otimes_{R_B}[R_B\otimes_{R_G}K_G(Y)] \ar[u]^{res_X\otimes res_Y}& R_B\otimes_{R_G}K_G(X\times Y)\ar[u]^{res_{X\times Y}}\\
[R_B\otimes_{R_G}K_G(X)]\otimes_{R_G}K_G(Y)]\ar[u]^{natural} \ar[r]^{id_{R_B}\otimes {\rm Ku}_G}& 
R_B\otimes_{R_G}K_G(X\times Y)\ar[u]^{\rm id}.\\
}
$$ 
The vertical maps are isomorphisms due to 
Theorem A.1, and ${\rm Ku}_B$ is an isomorphism 
by \cite[Proposition 6.4]{ap:opk} and the fact that the functors 
${\rm op}K_B(-)$ and $\K(-)$ agree on $B$-schemes. 
Therefore, the bottom horizontal map 
is also an isomorphism.   
But this morphism is a faithfully flat extension of ${\rm Ku}_G$, 
because $R(B)\simeq R(T)$ is a free $R(G)$-module. 
%
%
We conclude that ${\rm Ku}_G$ is an isomorphism of $R_G$-modules. 
\end{proof}

%

\medskip

From Theorem \ref{GequivKunneth.thm} one formally deduces, 
as in the $T$-equivariant case 
(cf. \cite[Proposition 6.3]{ap:opk}), 
the following. 

\begin{cor}\label{GequivKro.cor} 
Let $G$ be a connected reductive group with $\pi_1(G)$ torsion free.
If $X$ be a complete $G$-spherical variety,  
then the $G$-equivariant Kronecker duality map 
$$\KG(X)\longrightarrow{\rm Hom}_{R(G)}(K^G(X),R(G)).$$ 
is an isomorphism. \hfill $\square$
\end{cor}

As a byproduct, our main theorems on torus actions 
can be used to calculate the $G$-equivariant operational 
$K$-theory of spherical varieties. 

\begin{cor}\label{ring.W.inv.cor}
Let $G$ be a connected reductive group with $\pi_1(G)$ torsion free.
If 
$X$ is complete $G$-spherical variety, then 
$$\K(X)\simeq \KG(X)\otimes_{R(G)}R(T).$$  
Consequently, $(\K(X))^W\simeq \KG(X)$.  
\end{cor}

\begin{proof}
In light of Corollary \ref{GequivKro.cor} we have the isomorphism 
$$
\KG(X)\simeq {\rm Hom}_{R(G)}(K^G(X),R(G)).
$$
Since the $R(G)$-modules 
$R(T)$ and $K^G(X)$ are, respectively, 
free and finitely generated, tensoring with $R(T)$ 
both sides of the identity above yields 
$${\rm Hom}_{R(T)}(K^G(X)\otimes_{R(G)}R(T),
R(G)\otimes_{R(G)}R(T)).$$
The latter expression identifies, in turn, to 
$\K(X)$, due to 
$T$-equivariant Kronecker duality (Property (f), Introduction) 
and the fact that $K^G(X)\otimes_{R(G)}R(T)\simeq K^T(X)$. 
The second assertion follows from 
our previous argument once we 
recall that $R(T)\simeq R(G)\otimes \Z[W]$. 
\end{proof}

Next we show that 
the functors $\KG(-)$ and $K_G(X)$ agree 
on smooth projective $G$-spherical varieties. 
For this, a few extra facts need to be brought to our  
attention.  
The natural map 
$\KG(X\to pt)\to K^G(X)$ is always 
surjective. This stems from the proof of 
\cite[Proposition 4.4]{ap:opk}. 
%
%
In fact, this map is an isomorphism when $G$ is a torus.  
Nonetheless, for more general reductive groups $G$, 
injectivity of such map 
is a more delicate issue, and it does not 
follow from the arguments given in \cite[Proposition 4.4]{ap:opk}.  
The main problem is that, unlike to torus case,    
$K^G(X)$ might not be generated by the classes of the 
structure sheaves of $G$-invariant subvarieties. 
%
We can overcome this issue in the case 
of smooth projective $G$-spherical 
varieties by appealing to 
$G$-equivariant Kronecker duality. 
%

\begin{cor}
Let $G$ be a connected reductive group with 
$\pi_1(G)$ torsion free.
If $X$ is a smooth projective $G$-spherical 
variety, then the natural maps 
$$K_G(X)\to \KG(X)\to \KG(X\to pt)\to K^G(X)$$ 
are isomorphisms.   
\end{cor}

\begin{proof}
Note that the displayed diagram  
is a factorization of the Poincar\'e duality 
map $K_G(X)\to K^G(X)$, which is known to be 
an isomorphism. Moreover, 
$K_G(X)$ is a free $R(G)$-module of 
rank $|X^T|$ \cite[Lemma 1.6]{uma:kth}.  
Thus, 
it suffices 
to show that the last two maps in the array are 
isomorphisms. Bearing this in mind, observe that 
the map $\KG(X)\to \KG(X\to pt)$ 
is an isomorphism by \cite[Proposition 4.3]{ap:opk}, 
(which is independent of their Lemma 2.3, the main technical 
issue in our setting). 
On the other hand, by Corollary \ref{GequivKro.cor},  
$\KG(X)$ is a  
free $R(G)$-module 
of rank $|X^T|$ (hence so is $\KG(X\to pt)$).  
It follows that $\KG(X\to pt)\to K^G(X)$,  
being a surjective map of free modules of the same rank, 
is an isomorphism.   
\end{proof}




\end{document}